\documentclass[11pt]{amsart}
\usepackage{xcolor}
\usepackage{hyperref}
\hypersetup{
    colorlinks = false,
    linkbordercolor = {white},
    citebordercolor = {white}
    }
\usepackage[capitalise,noabbrev]{cleveref}
\usepackage[all]{xy}

\usepackage[utf8]{inputenc}
\usepackage{amsmath,amscd}
\usepackage{amssymb}
\usepackage{mathtools}
\usepackage{stmaryrd}
\usepackage{url}
\usepackage{tikz-cd}
\usepackage{enumitem}
\usepackage{fullpage}
\usepackage{musicography}

\usepackage{physics}

\usepackage{tikz}

\usepackage{enumitem,kantlipsum}

\usepackage{mathtools}

 \usepackage{comment}
 \usepackage{mathrsfs}
 \usepackage{tikz-cd}
\usepackage{relsize}

\author{Andrew Soto Levins}
\author{Prashanth Sridhar}

\newcommand{\Addresses}{{
	\vskip\baselineskip
  	\footnotesize
  	\noindent \textsc{Department of Mathematics and Statistics, Texas Tech University} \\ \textsc{Department of Mathematics and Statistics, Auburn University} \par\nopagebreak
	\noindent \textit{E-mail addresses:} \texttt{ansotole@ttu.edu, pzs0094@auburn.edu}
 }}

\numberwithin{equation}{section}
\newtheorem{lemma}[equation]{Lemma}
\newtheorem{lem}[equation]{Lemma}
\newtheorem{theorem}[equation]{Theorem}

\newtheorem{prop}[equation]{Proposition}

\newtheorem{claim*}{Claim}

\theoremstyle{definition}

\newtheorem{dfn}[equation]{Definition}

\newtheorem{constr}[equation]{Construction}

\newtheorem{setup}[equation]{Setup}

\theoremstyle{remark}
\newtheorem{notation}[equation]{Notation}
\newtheorem{remark}[equation]{Remark}

\newtheorem{rem}[equation]{Remark}

\newcommand{\mfrak}[1]{\mathfrak{#1}}

\newcommand{\msf}[1]{\mathsf{#1}}

\renewcommand{\H}{\mathrm{H}}

\newcommand{\m}{\mfrak{m}}
\newcommand{\n}{\mfrak{n}}

\newcommand{\injdim}{\operatorname{inj\,dim}}
\newcommand{\projdim}{\operatorname{proj\,dim}}
\newcommand{\flatdim}{\operatorname{flat\,dim}}

\newcommand{\depth}{\operatorname{depth}}
\newcommand{\amp}{\operatorname{amp}}
\newcommand{\Hom}{\operatorname{Hom}}
\newcommand{\RHom}{\mathrm{RHom}}

\newcommand{\Lotimes}{\otimes^\msf{L}}

\newcommand{\D}{\msf{D}}

\newcommand{\del}{\partial}
\newcommand{\Mod}{\operatorname{Mod}}
\def\mod{\operatorname{mod}}

\newcommand{\im}{\operatorname{Im}}

\newcommand{\dGamma}{\mathbf{R}\Gamma}

\def\nc{\newcommand}
\nc{\on}{\operatorname}
\nc{\bideg}{\on{bideg}}
\nc{\xra}{\xrightarrow}
\def\phi{\varphi}
\nc\cB{\mathcal{B}}

\def\bu{\bullet}
\def\D{\on{D}}
\def\Db{\D^{\on{b}}}
\def\Df{\D^{\on{f}}}
\nc{\into}{\hookrightarrow}
\nc{\onto}{\twoheadrightarrow}
\nc{\LL}{\mathbf{L}}
\nc{\RR}{\mathbf{R}}
\nc{\Perf}{\on{Perf}_{\on{gr}}}
\nc{\nat}{\natural}
\nc{\tors}{\on{tors}}
\nc{\Tors}{\on{Tors}}
\def\mod{\on{mod}}
\def\Mod{\on{Mod}}
\nc{\qgr}{\on{qgr}}
\nc{\Qgr}{\on{Qgr}}
\nc{\fQgr}{\on{Qgr}^{\on{f}}}
\nc{\colim}{\on{colim}}
\def\Z{\mathbb{Z}}
\nc{\Ext}{\on{Ext}}

\nc{\om}{\omega}
\nc{\w}{\widetilde}
\nc{\PP}{\mathbb{P}}
\nc{\mf}{\on{mf}}
\nc{\OO}{\mathcal{O}}
\nc{\Proj}{\on{Proj}}
\nc{\Qcoh}{\on{Qcoh}}
\nc{\coh}{\on{coh}}
\nc{\Tor}{\on{Tor}}
\nc{\Modf}{\Mod^{\on{f}}}

\nc{\ce}{\coloneqq}

\nc{\Com}{\on{Com}}
\nc{\A}{\mathcal{A}}
\nc{\B}{\mathcal{B}}
\nc{\C}{\mathcal{C}}
\nc{\I}{\mathcal{I}}
\nc{\M}{\mathcal{M}}
\nc{\Sh}{\on{Sh}}
\nc{\QCoh}{\on{Qcoh}}
\nc{\Coh}{\on{coh}}
\nc{\fQCoh}{\QCoh^{\on{f}}}
\nc{\ov}{\overline}
\nc{\End}{\on{\underline{End}}}
\def\MR#1{}

\nc{\Qgrf}{\Qgr^{\on{f}}}
\nc{\uHom}{\underline{\Hom}}
\nc{\Inj}{\mathrm{Inj}}
\nc{\proj}{\mathrm{Proj}}
\nc{\spec}{\mathrm{Spec}}
\nc{\xla}{\xleftarrow}
\nc{\Dqgr}{\D_{\qgr}}
\nc{\DQgr}{\D_{\Qgr}}
\nc{\cK}{\mathcal{K}}
\nc{\from}{\leftarrow}
\nc{\cd}{\on{cd}}
\nc{\N}{\mathcal{N}}

\usepackage{comment}

\begin{document}
\title{Module-Theoretic characterizations of  Gorenstein morphisms}

\begin{abstract}

The Gorenstein property in local algebra admits several characterizations via its module category. The goal of this paper is to collect and generalize such characterizations to the relative setting, i.e., to Gorenstein morphisms as defined by \cite{AF}. We achieve this by proving these characterizations more generally for graded-commutative Gorenstein dg-algebras.

\end{abstract}

\thanks{{\em Mathematics Subject Classification} 2020: 13H10.}

\numberwithin{equation}{section}

\maketitle

\setcounter{section}{0}

\section{Introduction}

The Gorenstein property is, roughly speaking, the algebraic analog of spaces with Poincar\'{e} duality. When this property is available, there are remarkable duality results in varying contexts: in cohomology theories, free resolutions and invariant theory to name a few, see for example \cite{SD_1960-1961__14_1_A2_0,bass, Hartshorne1967-me,stanley,gorenstein_spaces}. General duality statements are often much simpler in this case - for instance, the dualizing complex of a Gorenstein scheme is a line bundle. The Gorenstein property results in many nice consequences for the module category, and in return, it can be characterized purely in terms of its module category, see \cite{bass_injective,bass, Auslander, Peskine_Szpiro, syzygies,buchweitz, MCM_approximations,Christensen2000-fv, Avramov_martsinkovsky} for numerous such applications. This property has also played an important role in noncommutative algebra and in particular in the representation theory of Artin algebras. We refer the reader to \cite{huneke_gorenstein} for more about its utility. 

\par Avramov-Foxby in \cite{AF} developed a relative theory of the Gorenstein property in local algebra: a morphism of commutative Noetherian local rings $(R,\m)\to (S,\n)$ is Gorenstein if $\flatdim_R(S)<\infty$ and there exists an integer $d\in \mathbb{Z}$ such that the Bass numbers (\cite{Bass_Gorenstein})  satisfy $\mu^i_R(\m,R)=\mu^{i+d}_S(\n,S)$ for all $i\in \mathbb{Z}$. Since Gorenstein rings can be characterized by the property that they have the nicest Bass numbers, Avramov-Foxby's definition is taking a relative perspective by asking for the nicest relationship between Bass numbers along a ring homomorphism. A local ring $S$ is Gorenstein if and only if the structure morphism $\mathbb{Z}_{(p)}\to S$ is Gorenstein, where $p$ is the characteristic of the residue field of $S$ (\cite{AF}, Proposition 2.1). This definition is also an extension of Grothendieck's definition of a Gorenstein morphism i.e., flat morphisms with Gorenstein fibers.   

\par On the other hand, the Gorenstein condition has been developed and used in higher algebra, see \cite{gorenstein_spaces, Frankild2003,FIJ,DGI_duality0, yekutieli2013duality,SAG_Lurie, Shaul_CM, hu2023gdimensionsdgmodulescommutativedgrings, BROWNSRIDHAR, brown2024serredualitydgalgebras}. The connection between these two generalizations is that the relative theory can be characterized in terms of the corresponding notion in higher algebra, see \cite[Theorem 2.2]{AF}. The goal of this paper is to collect and generalize to the relative setting, several module-theoretic characterizations of the Gorenstein condition in local algebra. We achieve this by exploiting the recent developments in the theory of differentially graded (dg) algebras - more specifically in the theory of non-positive, graded commutative, dg-algebras in \cite{Shaul_CM, minamoto, Minamoto2,BSSW,local_CM_modules,hu2023gdimensionsdgmodulescommutativedgrings}. Below are two examples of classical module theoretic statements this paper generalizes to the relative setting:

\begin{theorem}\cite{Foxby_iso}\label{fact_1}
    A commutative Noetherian local ring $R$ is Gorenstein if and only if there exists a non-zero finitely generated $R$-module $M$ such that $\injdim_R(M)<\infty$ and $\projdim_R(M)<\infty$.
\end{theorem}

\begin{theorem}\label{fact_2}
   A commutative Noetherian local ring $R$ is Gorenstein if and only if there exists a non-zero finitely generated $M$ such that $\injdim(M)<\infty$, $M$ has finite Auslander bound (\Cref{def:Auslander_Bounds}) and $\Ext^i(M,R)=0$ for $i\gg 0$.
\end{theorem}

The Auslander bound is an invariant arising from an important conjecture of Auslander concerning finite dimensional algebras. It was developed by Wei in \cite{Auslander_bounds_wei} and subsequently by many others. For its connections with the Auslander-Reiten conjecture, see \cite{Auslander_bounds_wei}. Note that the characterization in \Cref{fact_2} apriori appears weaker than the one in \Cref{fact_1}. Several such module-theoretic characterizations are scattered across the literature. Our relative statement is as follows:

\begin{theorem}[Special case of \Cref{thm:main}] \label{theorem1}
   Let $\Phi: (R,\m,k)\to (S,\n,l)$ be a map of commutative noetherian local rings such that $\flatdim_R(S)<\infty$. Let $\mathrm{F}:=k\Lotimes_R S$ be the derived fiber (\Cref{def:derived_fiber}). Then, the following are equivalent:

   \begin{enumerate}
       \item $\Phi$ is Gorenstein.

       \item There exists $0\not\simeq M\in \Db(\mathrm{F})$ such that $\injdim_{\mathrm{F}}(M)<\infty$ (\Cref{def:injective}) and $\projdim_{\mathrm{F}}(M)<\infty$ (\Cref{def:projective}).

            \item There exists $0\not\simeq M\in \Db(\mathrm{F})$ such that (a) $\injdim_{\mathrm{F}}(M)<\infty$ (\Cref{def:injective}) \\ (b) $\RHom_{\mathrm{F}}(M,\mathrm{F})\in \Db(\mathrm{F})$ and (c) $B(M)<\infty$ (\Cref{def:Auslander_Bounds}).

       \item For every $0\not\simeq M\in \Db(\mathrm{F})$, there exists $0\not\simeq N_M\in \Db(\mathrm{F})$ such that $N_M\in \langle M\rangle \cap \langle \mathcal{G}_0\rangle$, where $\mathcal{G}_0$ is the full subcategory of Gorenstein projective dg-modules (\Cref{def:gorenstein_projective}) and $\langle  -\rangle$ denotes the thick subcategory generated by a collection of objects.
       
       \item For every $M\in \Db(\mathrm{F})$ such that $\injdim_{\mathrm{F}}(M)<\infty$ (\Cref{def:injective}), there exists $0\not \simeq N_M\in \langle M\rangle \cap \langle l\rangle $ such that $B(N_M)<\infty$ and $\RHom_{\mathrm{F}}(M,\mathrm{F})$ is cohomologically bounded above, where $\langle  -\rangle$ denotes the thick subcategory generated by a collection of objects.

        \item There exists a dg module $0\not\simeq M\in \mathrm{MCM}(\mathrm{F})$ (\Cref{def:CM_modules}) with $\injdim_{\mathrm{F}}(M)<\infty$ (\Cref{def:injective}), $B(M)<\infty$ (\Cref{def:Auslander_Bounds}), and $\sup(M)=\sup\RHom_{\mathrm{F}}(M,\mathrm{F})=0$.

        \item Let $\underline{x}=x_1,\dots,x_n$ be a generating set for $\n$, and set $K:=\mathrm{F}//\underline{x}$ (\Cref{derived_KC}). There exists $0\not\simeq M\in \Db(K)$ with semifree resolution $G\xra{\simeq} M$ such that 

          \begin{enumerate}
              \item $G/F^nG$ has cohomology concentrated in a single degree for some integer $n$.

              \item $\RHom_K(M,K)\in \Db(K)$.
          \end{enumerate}
    \end{enumerate}

\end{theorem}

The statement (4) above extends the characterization of the Gorenstein property in terms of Gorenstein dimension and Gorenstein projective modules, (5) in terms of virtual smallness, (6) in terms of maximal Cohen-Macaulay modules and (7) in terms of the derived category of the Koszul complex associated to the maximal ideal. We prove \Cref{theorem1} by establishing the equivalence of the assertions more generally for non-positive, graded commutative, cohomologically bounded, dg-algebras (\Cref{setup}) in \Cref{thm:main}. This is indeed more general due to \cite[Theorem 2.2]{AF}. As noted earlier, the proofs of these results rely on recent developments in the theory of resolutions and homological dimensions of such dg-algebras. The classical characterizations for a Gorenstein ring $S$ can be recovered by taking $R=\mathbb{Z}_{(p)}$ and $\Phi$ to be the structure morphism in \Cref{theorem1}, where $p$ is the residue characteristic of $S$. For instance, statements (2) and (3) in \Cref{theorem1} recover \Cref{fact_1} and \Cref{fact_2} under this hypothesis.

\par The paper has two sections: in section 2, we collect all the notions from the literature we need and in section 3, we prove \Cref{thm:main}, our main result.

\subsection*{Acknowledgments.} We thank Michael K. Brown, Nawaj KC, Zach Nason, Josh Pollitz, Mark Walker and Ryan Watson for helpful conversations and comments on a preliminary draft.

\section{Preliminaries}

In this section, we collect from various sources, the definitions and notions we need from the theory of differential graded algebras.

\begin{notation}
\label{conventions}
A \emph{differential graded (dg) algebra} is a graded algebra $A = \bigoplus_{j \in \Z} A^j$ equipped with a degree $1$ $\mathbb{Z}$-linear map $\del_A$ that squares to 0 and satisfies the Leibniz rule:
$$
\del_A(xy) = \del_A(x)y + (-1)^{\deg{x}}x\del_A(y).
$$
Say $A$ is \emph{graded commutative} if $xy = (-1)^{|x||y|}yx$ for all homogeneous $x, y \in A$ and $x^2=0$ if $\deg(x)$ is odd. A \emph{morphism} of dg algebras is a degree $0$ algebra homomorphism that commutes with differentials and is a \emph{quasi-isomorphism} if it induces an isomorphism on cohomology. We denote by $A^{\nat}$ the underlying graded algebra of a dg-algebra $A$. We say a dg-algebra $A$ is non-positive if $A^i=0$ for $i>0$.
\par
Let $A$ be a dg-algebra. A right (resp. left) \emph{dg-$A$-module} is a graded right (resp. left) $A^\nat$-module $M = \bigoplus_{j \in \Z} M^j$ equipped with a degree $1$ $\mathbb{Z}$-linear map $\del_M$ that squares to 0 and satisfies the Leibniz rule:
$$
\del_M(mx) = \del_M(m)x + (-1)^{\deg(m)}m\del_A(x)
\quad
(\text{resp. }\del_M(xm) = \del_A(x)m + (-1)^{\deg(x)}x\del_M(m)).
$$

 If $A$ is graded commutative and if $M$ is a right dg $A$-module, then $M$ is also a left dg $R$-module with left action $rm \ce (-1)^{\deg(r) \deg(m)} mr$ for homogeneous elements $r \in R$ and $m \in M$. This gives $M$ a dg $A$-$A$-bimodule structure and induces an equivalence between the category of right and left dg-modules.

 \par A dg-module $M$ is said to be \emph{finitely generated} if the underlying graded $A^{\nat}$-module is finitely generated. A \emph{morphism} of dg-$A$-modules is a degree $0$ $A^{\nat}$-linear map that commutes with differentials. Such a morphism is a \emph{quasi-isomorphism} if it induces an isomorphism on cohomology. 


Given dg-$A$-modules $M$ and $N$, the tensor product $M \otimes_A N$ is a dg-$A$-module with differential $m \otimes n \mapsto d_M(m) \otimes n + (-1)^{|m|}m \otimes d_N(n)$. Similarly, given dg-$A$-modules $M$ and $N$, we may form the internal Hom object $\Hom_A(M, N)$, which is a dg-$A$-module with underlying $A^{\nat}$-module $\Hom_{A^{\nat}}(M, N)$ and differential $\alpha \mapsto d_N \alpha - (-1)^{|\alpha|}\alpha d_M$. A map in $\Hom_A(M, N)$ of degree $0$ is a cocycle if and only if it is a morphism of dg-$A$-modules.

\end{notation}

For a dg-algebra $A$, let $\Mod(A)$ denote the category of dg-$A$-modules, $\Modf(A)$ the full subcategory of $\Mod(A)$ given by dg-$A$-modules $M$ such that $H(M)$ is finitely generated over $H(A)$, and $\mod(A)$ the full subcategory of $\Modf(A)$ given by dg-$A$-modules that are finitely generated over $A$. We form the (triangulated) derived categories $\D(A)$, $\Df(A)$, and $\D^b(A)$ by inverting quasi-isomorphisms in $\Mod(A)$, $\Modf(A)$, and $\mod(A)$, respectively. For a construction of the derived category of a dg-algebra, see \cite[\href{https://stacks.math.columbia.edu/tag/09KV}{Tag 09KV}]{stacks-project}. See \cite{yekutieli_2019} for a comprehensive treatment of derived categories of dg-algebras.

\par Let $\D^+(A)$ (resp. $\D^-(A)$) denote the full subcategory of $\D(A)$ of dg-modules $M$ such that $\H^i(M)=0$ for $i\ll 0$ (resp. $\H^i(M)=0$ for $i\gg 0$). For a dg-module $M$, write 

\[\inf(M):=\inf\{i\:|\: \H^i(M)\neq 0\}.\]

\[\sup(M):=\sup\{i\:|\: \H^i(M)\neq 0\}. \]

\[\amp(M)=\sup(M)-\inf(M).\]

Given a dg-module $M$, its $i$-th cohomological shift, $M[i]$, is the dg-module with underlying module given by $M[i]^j=M^{i+j}$ and differential $d_{M[i]}=(-1)^id_M$.

\par Note that there is a natural map of dg-algebras $A\to \H^0(A)$ and $A\to \mathbf{k}$, where $\mathbf{k}$ is the residue class field of $\H^0(A)$.

\begin{remark}
\label{remark:smart}
Let $A$ be a non-positive dg-algebra and $M$ a dg-$A$-module. For $i \in \Z$, we have \emph{smart truncations}
\begin{align*}
\sigma^{\ge i} M & \ce \left( \cdots \to 0 \to M^i / \im(d_M^{i - 1}) \to M^{i+1} \to \cdots \right), \\ 
\sigma^{\le i} M & \ce \left( \cdots \to M^{i-1} \to\ker(d^i_M) \to 0 \to \cdots \right).
\end{align*}
Both $\sigma^{\ge i} M$ and $\sigma^{\le i} M$ are dg-$A$-modules and the natural map $M \to \sigma^{\ge i} M$ (resp. $\sigma^{\le i} M \to M$) induces an isomorphism on cohomology in degrees at least $i$ (resp. at most $i$). 
\end{remark}

\begin{dfn}\label{def:derived_fiber}\cite{Avramov_golod}
     Let $\Phi: (R,\m,k)\to (S,\n,l)$ be a map of commutative noetherian local rings. The \textit{derived fiber} of $\Phi$ is the non-positive, graded commutative dg-algebra $k\Lotimes_R S:= k\otimes_R G$, where $G\xra{\simeq} S$ is a quasi-isomorphism of dg-algebras over $R$ with $G$ non-positive, graded commutative and degree wise free over $R$. The derived fiber always exists and the definition is unique up to a zig-zag of quasi-isomorphisms of dg-algebras. 
\end{dfn}

\begin{dfn}
    Let $A$ be a dg-algebra. A dg-$A$-module $G$ is \emph{free} if it is isomorphic, as a dg-module, to a direct sum of copies of $A[i]$ for various $i \in \Z$.
A dg-$A$-module $G$ is called \emph{semi-free} if it can be equipped with an increasing, exhaustive filtration $F^\bu G$ by dg-submodules such that $F^i G = 0$ for $i < 0$ and each dg-module $F^i G / F^{i-1} G$ is free. Given a dg-$A$-module $M$, a \emph{semi-free resolution} of $M$ is a quasi-isomorphism $G \xra{\simeq} M$, where $G$ is semi-free. It is well-known that every dg-$A$-module $M$ admits a semifree resolution; see for example \cite[Chapter 5, Theorem 2.2]{avramov1997differential}.
\end{dfn}

We now recall a dg-analog of projective resolutions due to Minamoto in \cite{minamoto}. Let $\mathcal{P}=\mathrm{add}(A)$ be the additive closure of $A$ in $\D(A)$.

\begin{dfn}[\cite{minamoto} - Definitions 2.12 and 2.17]\label{def:sppj}
Let $A$ be a non-positive dg-algebra and $M\in \D^-(A)$.

\begin{enumerate}
    \item A sppj morphism is a morphism $f\in \Hom_{\D(A)}(P,M)$ such that (i) $P\in \mathcal{P}[-\sup(M)]$ and (ii) $\H^{\sup}(f)$ is surjective.

    \item A sppj resolution $P$ of $M$ is a sequence of exact triangles for each $i\geq 0$ $
\left[M_{i+1} \xra{g_{i+1}} P_i \xra{f_i} M_{i} \to \right]$ such that $M_0=M$ and $f_i$ is an sppj morphism.
 
\end{enumerate}
    
\end{dfn}

\begin{rem}
    Note that in an sppj resolution we have $\sup(M_{i+1})=\sup(P_{i+1})\leq \sup (P_i)=\sup(M_i)$ for all $i\geq 0$.
\end{rem}

We now recall the definition of projective and injective dimension for dg-modules.

\begin{dfn}[\cite{BSSW}]\label{def:projective}
    Let $A$ be a non-positive dg-algebra and $M\in \D(A)$. The \textit{projective dimension of $M$} is defined as

    \[\projdim_A(M):=\inf\{n\in \mathbb{Z}\:|\: \Ext^i_A(M,N)=0 \text{ for any } N\in \Db(A) \text{ and any } i>n+\sup(N) \}.\]
\end{dfn}

\begin{dfn}[\cite{BSSW}]\label{def:injective}
    Let $A$ be a non-positive dg-algebra and $M\in \D(A)$. The \textit{injective dimension of $M$} is defined as

    \[\injdim_A(M):=\inf\{n\in \mathbb{Z}\:|\: \Ext^i_A(N,M)=0 \text{ for any } N\in \Db(A) \text{ and any } i>n-\inf(N) \}.\]
\end{dfn}

\begin{rem}
    In \cite{yekutieli2013duality}, another definition of projective and injective dimension was made, denoted $\mathrm{pd}_A(M)$ and $\mathrm{id}_A(M)$ respectively. The relationship between the two definitions is as follows: (a) if $M\in \D^-(A)$, then $\projdim_A(M)=\mathrm{pd}_A(M)-\sup(M)$ (b) if $M\in \D^+(A)$, then $\injdim_A(M)=\mathrm{id}_A(M)+\inf(M)$.
\end{rem}

 \begin{setup}\label{setup}
   All the main results in this paper are written in the following setup. The symbol $A$ will denote a non-positive graded commutative dg-algebra i.e a graded commutative dg-algebra such that $A^i=0$ for $i>0$. Further, assume $(\H^0(A),\m)$ is Noetherian, local and that $\H(A)$ is finitely generated over $\H^0(A)$.
\end{setup}

We recall a notion of a dualizing dg-module due to Yekutieli.

\begin{dfn}[\cite{yekutieli2013duality}]\label{def:dualzing1}
Let $A$ be a non-positive, graded commutative dg-algebra. An object $R\in \D^+(A)$ is \textit{dualizing} if 

\begin{enumerate}
    \item for each $i\in \mathbb{Z}$, $\H^i(R)$ is a finitely generated $\H^0(A)$ module. 

    \item $\injdim_A(R)<\infty$.

    \item the homothety morphism $A\to \RHom_A(R,R)$ is an isomorphism in $\D(A)$.
\end{enumerate}
\end{dfn}

In \cite{FIJ}, another definition of a dualizing dg-module was studied. An object $M\in \D(A)$ is $R$-reflexive for an object $R\in \D(A)$ if the derived evaluation morphism of $M$ with respect to $R$ is a quasi-isomorphism. If $M$ is $A$-reflexive, we say $M$ is reflexive.

\begin{dfn}[\cite{FIJ}]\label{def:dualizing2}
Let $A$ be as in \Cref{setup}. An object $R\in \D^+(A)$ is \textit{dualizing} if for all $M\in \Db(A)$

\begin{enumerate}
    \item $\RHom_A(M,R)\in \Db(A)$.
    \item $M$ and $M\Lotimes_A R$ are $R$-reflexive. 
\end{enumerate}
\end{dfn}

\begin{rem}
    Note that the two definitions for a dualizing dg-module above are equivalent in the setting of \Cref{setup}. This can be seen from \cite[Theorem 3.3]{Minamoto2} and \cite[Theorem 3.2 and Proposition 3.4]{FIJ}.
\end{rem}

\begin{rem}
    If $A\to B$ is a morphism of graded commutative dg-algebras and $R$ is a dualizing dg-module for $A$, then $\RHom_A(B,R)$ is a dualizing dg-module for $B$. This follows from \cite[Theorem 3.2]{FIJ} and adjunction.
\end{rem}

\begin{rem}
Let $A$ be as in \Cref{setup}. If $A$ is cohomologically $\m$-adically complete, then $A$ admits a dualizing dg-module by \cite[Proposition 7.21]{Shaul_injective}. In particular, the derived completion $\mathbb{L}\Lambda(A,\m)$ has a dualizing dg-module.
\end{rem}

Recall the definition of a Gorenstein dg-algebra:

\begin{dfn}(Frankild-Iyengar-Jorgensen, Yekutieli)\label{def:Gorenstein}
Let $A$ be as in \Cref{setup}. Then, $A$ is a \textit{Gorenstein dg-algebra} if any of the following equivalent conditions hold:
\begin{enumerate}
    \item $A$ is a dualizing dg-module for $A$.
    \item $\injdim_A(A)<\infty$.
    \item for all $M\in \Db(A)$, $\RHom_A(M,A)\in \Db(A)$ and the natural map $M\to \RHom_A(\RHom_A(M,A),A)$ is a quasi-isomorphism.
    \item the functor $\RHom_A(-,A)$ is an anti-equivalence on $\Db(A)$.
    \item there is an isomorphism $\RHom_A(k,A)\simeq k$ in $\D(A)$ where $k$ is the residue field of $\H^0(A)$.
\end{enumerate}
    
\end{dfn}

The following is the analog of a Koszul complex for a graded commutative dg-algebra.

\begin{constr}[\cite{Minamoto2}, \cite{Shaul_CM}]\label{derived_KC}
  Let $A$ be as in \Cref{setup} and $x\in \H^0(A)$ an element. Under the identification $\H^0(A)\simeq \H^0(\RHom(A,A))\simeq \Hom_{\D(A)}(A,A)$, $x$ corresponds to an endomorphism of $A$ in $\D(A)$. Denote the cone of this morphism by $A//x$. It is shown in \cite[Section 2.2]{Minamoto2} that  $A//x$ has the structure of a graded-commutative dg-algebra. Now, if $\underline{x}=x_1,\dots,x_n$ is a sequence of elements in $\H^0(A)$, then $A//\underline{x}$ is defined inductively as $(A//(x_1,\dots,x_{n-1}))//x_n$. Moreover, we have $\H^0(A//\underline{x})=\H^0(A)/(\underline{x})$ and that $\H(A//\underline{x})$ is finitely generated over $\H^0(A//\underline{x})$. In particular, $\H(A//\underline{x})$ has finite length if $\H^0(A//\underline{x})$ is Artinian.  
\end{constr}

Analogs of Gorenstein dimension for dg-modules and Gorenstein projective modules over dg-algebras were introduced recently in \cite{hu2023gdimensionsdgmodulescommutativedgrings}.

\begin{dfn}[\cite{hu2023gdimensionsdgmodulescommutativedgrings}]\label{def:gorenstein_projective}
Let $A$ be as in \Cref{setup}. Let $M$ be a reflexive dg-module such that $\RHom_A(M,A)\in \Db(A)$.
    \begin{enumerate}
        \item The \textit{Gorenstein dimension} of $M$, denoted $\mathrm{Gdim}_A(M)$, is the integer $\sup \RHom_A(M,A)$.

        \item $M$ is in the class $\mathcal{G}$ if either $\mathrm{Gdim}_A(M)=-\sup(M)$ or $M=0$. $M$ is in the class $\mathcal{G}_0$ if it is an object of the full subcategory of $\mathcal{G}$ consisting of the zero object and objects with the additional property $\mathrm{Gdim}_A(M)=-\sup(M)=0$. The class $\mathcal{G}_0$ consists of the \textit{Gorenstein projective} dg-$A$-modules.
    \end{enumerate}
\end{dfn}

We extend the definition of Auslander Bounds for modules and complexes to dg-modules.

\begin{dfn}[cf. \cite{Auslander_bounds_wei} - Definition 2.1, \cite{levins2024studyauslanderbounds} - Definition 6.2]\label{def:Auslander_Bounds}
 Let $A$ be as in \Cref{setup}. Let $M,N\in \Db(A)$ with $\sup(N)<\infty$. Set

 \[P_A(M,N):=\sup\{n\:|\: \Ext_A^{n+\sup(N)}(M,N)\neq 0\}.\]

 The \textit{Auslander Bound} $B(M)$ of $M$ is defined as

 \[B(M):=\sup \{P_A(M,N)\:|\: 0\not\simeq N\in \Db(A) \text{ and } P_A(M,N)<\infty\}.\]
    
\end{dfn}

\begin{rem}
 \noindent \begin{enumerate}
      \item In \Cref{def:Auslander_Bounds}, if $A,M,N$ are concentrated in cohomological degree $0$, then the definition recovers the usual definition of Auslander Bounds for modules over rings.

      \item Note that for any $k\in \mathbb{Z}$, $P_A(M,N)=P_A(M,N[k])$.

      \item Suppose $A$ is a dg-algebra as in \Cref{setup}. It follows from the definitions that if $\projdim_A(M)<\infty$, then $B(M)=\projdim_A(M)$.
  \end{enumerate}

\end{rem}

There exist natural extensions of fundamental notions in commutative algebra such as localization, support and depth to dg-modules over non-positive graded commutative dg-algebras.

\begin{dfn}[\cite{yekutieli2013duality}]
    Let $A$ be a non-positive graded commutative dg-algebra. For a prime ideal $\bar{\mathfrak{p}}$ of $\H^0(A)$, the \textit{localization} $A_{\bar{\mathfrak{p}}}$ is defined as follows. Let $\pi:A^0\to \H^0(A)$ denote the canonical projection, and set $\mathfrak{p}=\pi^{-1}(\bar{\mathfrak{p}})$. Then $A_{\bar{\mathfrak{p}}}$ is defined as $A\otimes_{A^0}A^0_{\mathfrak{p}}$. If $M\in \D(A)$, then $M_{\bar{\mathfrak{p}}}:=M\otimes_AA_{\bar{\mathfrak{p}}}$.
\end{dfn}

\begin{dfn}[\cite{Shaul_CM}]
      Let $A$ be a non-positive graded commutative dg-algebra. For $M\in \D(A)$, the \textit{support} of $M$ is defined as

     \[\mathrm{Supp}(M):=\{\bar{\mathfrak{p}}\in \spec(\H^0(A))\:|\: M_{\bar{\mathfrak{p}}}\not\simeq 0\}.\]

     In other words, $\mathrm{Supp}(M)=\bigcup_{n\in \mathbb{Z}}\mathrm{Supp}(\H^n(M))$.
\end{dfn}

\begin{dfn}[\cite{Shaul_CM}]

For a dg-module $M$, the \textit{depth of $M$} is the integer $\inf\RHom_A(\mathbf{k},M)$.
    
\end{dfn}

A theory of local cohomology for dg-algebras exists by virtue of a Brown representability theorem for triangulated categories due to Neeman, \cite{neeman}. Shaul identified and developed this theory for graded commutative dg-algebras, while also giving explicit descriptions in this case, see \cite{Shaul_completion_torsion, Shaul_CM}. See also \cite{brown2024serredualitydgalgebras} for the noncommutative homogeneous or bigraded case.

\begin{dfn}(\cite{Shaul_completion_torsion})
Let $A$ be a dg-algebra as in \Cref{setup}. Let $\bar{I}\subseteq \H^0(A)$ be an ideal and $\D_{\bar{I}-tor}(A)\subseteq \D(A)$ be the full triangulated subcategory of dg-$A$-modules $M$ such that $\H^i(M)$ is $\bar{I}$-torsion for all $i$. Then by \cite[Theorem 8.4.4]{neeman}, the inclusion $i:\D_{\bar{I}-tor}(A)\to \D(A)$ has a triangulated right adjoint $G:\D(A)\to \D_{\bar{I}-tor}(A)$. The \textit{local cohomology} functor with respect to $\bar{I}$, $\dGamma_{\bar{I}}(-):\D(A)\to \D(A)$, is then defined as the composition $i\circ G$. For an explicit construction of this functor in terms of a weakly proregular resolution and a \v{C}ech like construction, see \cite{Shaul_completion_torsion}.
    
\end{dfn}

\begin{dfn}(\cite{foxby_flat},\cite{Shaul_CM})
Let $A$ be as in \Cref{setup} and $M\in \D^{-}(A)$. The local cohomology Krull dimension of $M$ is the number

\[\mathrm{lc. dim}(M):=\sup_{n\in \mathbb{Z}}\{\dim(\H^n(M))+n\}.\]

\end{dfn}

\begin{dfn}(\cite{Shaul_CM})\label{def:CM_modules}
Let $A$ be as in \Cref{setup} and $M\in \Db(A)$. 

\begin{enumerate}
    \item  $M$ is a \textit{local Cohen-Macaulay} dg-module if 

    \[\amp(M)=\amp(A)=\amp(\dGamma_{\m}(M)).\]
Denote by $\mathrm{CM}(A)$ the full subcategory of $\Db(A)$ consisting of local Cohen-Macaulay dg-modules.

\item $M$ is a \textit{maximal local Cohen-Macaulay} dg-module if $M\in \mathrm{CM}(A)$ and $\mathrm{lc. dim}(M)=\sup(M)+\dim(\H^0(A))$. Denote by $\mathrm{MCM}(A)$ the full subcategory of $\mathrm{CM}(A)$ consisting of maximal local Cohen-Macaulay dg-modules.
\end{enumerate}

\end{dfn}

\section{Main Result}

In this section, we prove the main result:

\begin{theorem}\label{thm:main}
    Let $A$ be a dg-algebra as in \Cref{setup}. The following are equivalent:

    \begin{enumerate}
        \item $A$ is Gorenstein.
        \item For every $0\not\simeq M\in \Db(A)$, there exists $0\not\simeq N_M\in \Db(A)$ such that $N_M\in \langle M\rangle \cap \langle \mathcal{G}_0\rangle$, where $\mathcal{G}_0$ is the full subcategory of Gorenstein projective dg-modules and $\langle  -\rangle$ denotes the thick subcategory generated by a collection of objects.

        \item There exists $0\not\simeq M\in \Db(A)$ such that $\injdim_A(M)<\infty$ and $\projdim_A(M)<\infty$.
        
        \item There exists a non-zero dg-module $M\in \Db(A)$ such that (a) $\injdim_{A}(M)<\infty$ \\ (b) $\RHom_A(M,A)\in \Db(A)$ and (c) $B(M)<\infty$.

        \item Let $k$ denote the residue field of $\H^0(A)$. For every $M\in \Db(A)$ such that $\injdim_A(M)<\infty$, there exists $0\not \simeq N_M\in \langle M\rangle \cap \langle k\rangle $ such that $B(N_M)<\infty$ and $P_A(N_M,A)<\infty$, where $\langle  -\rangle$ denotes the thick subcategory generated by a collection of objects.

        \item There exists a dg module $0\not\simeq M\in \mathrm{MCM}(A)$ with $\injdim_A(M)<\infty$, $B(M)<\infty$, and $\sup(M)=\sup(\RHom_A(M,A))=0$.

        \item Let $\underline{x}=x_1,\dots,x_n$ be a generating set for the maximal ideal of $\H^0(A)$, and set $K:=A//\underline{x}$. There exists $0\not\simeq M\in \Db(K)$ with semifree resolution $G\xra{\simeq} M$ such that 

          \begin{enumerate}
              \item $G/F^nG$ has cohomology concentrated in a single degree for some integer $n$.

              \item $\RHom_K(M,K)\in \Db(K)$.
          \end{enumerate}
    \end{enumerate}
\end{theorem}

\textbf{Proof of (1) $\iff$ (2).}

The following is a key technical result. The interested reader may also consult \cite[Theorem 6.2]{DGI}.

\begin{prop}\label{prop:gdim}
    Let $A$ be as in \Cref{setup}. If there exists $M\in \Db(A)$ such that (i) $\injdim_A(M)<\infty$ and (ii) there exists $0\neq N_M\in \langle M\rangle$ such that $\mathrm{Gdim}_A(N_M)<\infty$, then $A$ is Gorenstein.
\end{prop}

\begin{proof}
    Let $k$ denote the residue field of $\H^0(A)$. We will show that $\injdim_A(A)<\infty$. \cite[Corollary 2.31]{Minamoto2} shows that $\injdim_A(A)=\sup \RHom_A(k,A)$ and we show that $\RHom_A(k,A)$ is cohomologically bounded above. We calculate

    \begin{align*}
\RHom_A(k,N_M) &\simeq \RHom_A(k,\RHom_A(\RHom_A(N_M,A),A)) \\
&\simeq \RHom_A(\RHom_A(N_M,A)\Lotimes_Ak,A) \\
& \simeq \RHom_A(\RHom_A(N_M,A)\Lotimes_Ak\Lotimes_k k,A) \\
& \simeq \RHom_k(\RHom_A(N_M,A)\Lotimes_Ak,\RHom_A(k,A))\\
\end{align*} 

where the first isomorphism is due to the reflexivity of $N_M$ and the second and fourth are adjunctions. The object $\RHom_A(N_M,A)\Lotimes_Ak$ viewed as a complex of vector spaces is quasi-isomorphic to its cohomology, which is a direct sum of shifts of $k$. We thus have 

 \begin{align*}
\RHom_A(k,N_M) &\simeq \Hom_k(\bigoplus k[n]^{\oplus b_n},\RHom_A(k,A)) \\
&\simeq \prod \Hom_k(k[n]^{\oplus b_n},\RHom_A(k,A)) \\
& \simeq \prod \RHom_A(k,A)[-n]^{\oplus b_n}. \\
\end{align*} 

Since $\injdim_A N_M<\infty$, we see that $\sup \RHom_A(k,N_M)<\infty$. The above quasi-isomorphism then implies that $\sup \RHom_A(k,A)<\infty$ and hence that $\injdim_A(A)<\infty$ by \cite[Corollary 2.31]{minamoto}, finishing the proof.
    
\end{proof}

\begin{prop}\label{prop:G_0}
    Let $A$ be as in \Cref{setup}. Then $A$ is Gorenstein if and only if for every $0\not\simeq M\in \Db(A)$, there exists $0\not\simeq N_M\in \langle M\rangle \cap \langle \mathcal{G}_0\rangle$.
\end{prop}

\begin{proof}
The forward implication follows from \cite[Theorem 1.2]{hu2023gdimensionsdgmodulescommutativedgrings}. In more detail, \cite[Theorem 1.2(3)]{hu2023gdimensionsdgmodulescommutativedgrings} implies $\mathrm{Gdim}_A(M)<\infty$ for all $M\in \Db(A)$ if $A$ is Gorenstein. Applying \cite[Theorem 1.2(1) and (2)]{hu2023gdimensionsdgmodulescommutativedgrings} then gives the desired conclusion. To prove the converse, let $\bar{E}$ denote the injective hull of the residue field of $\H^0(A)$ over $\H^0(A)$. By Brown representability, there exists an object $E\in \D(A)$ such that for all $M\in \D(A)$, $\flatdim_A(M)=\injdim_A(\RHom_A(M,E))$ and $\H^i(\RHom_A(M,E))\simeq \Hom_{\H^0(A)}(H^{-i}(M),\bar{E})$ for all $i\in \mathbb{Z}$, see \cite[Definition 2.1, Lemma 7.12 and the discussion above it]{BSSW}. Setting $K:=A//\underline{x}$ for $\underline{x}$ a minimal generating set for the maximal ideal of $\H^0(A)$, we see that $B:=\RHom_A(K,E)\in \Db(A)$ and $\injdim_A(B)<\infty$. Choose a nontrivial $N_B\in \langle B\rangle \cap \langle \mathcal{G}_0\rangle$. Since the subcategory of reflexive modules of $\D(A)$ is a thick subcategory and $N_B\in \langle \mathcal{G}_0\rangle$, $N_B$ is reflexive. Moreover, if $X\to Y\to Z\to X[1]$ is a triangle in $\Db(A)$ with $X,Y,Z$ reflexive, then the inequality $\mathrm{Gdim}_A(Y)\leq \mathrm{sup}\{\mathrm{Gdim}_A(X), \mathrm{Gdim}_A(Z)\}$ implies that $\mathrm{Gdim}_A(N_B)<\infty$. Applying \Cref{prop:gdim} with $M=N_B$ finishes the proof.
\end{proof}

\textbf{Proof of (1) $\iff$ (4).}

\begin{lemma}\label{lem:fpd}
    Let $A$ be as in \Cref{setup}. Suppose $M\in \Db(A)$ is such that the quantities $B(M), P_A(M,A)$ and $P_A(M,M)$ are all finite, then so is $\projdim_A(M)$.
\end{lemma}

\begin{proof}
    We may assume that $\sup(M)=0$. For any sppj resolution $P$ of $M$ as in \Cref{def:sppj}, a standard two out of three argument shows that for all $i\geq 0$, $B(M_i), P_A(M_i,A),P_A(M_i,M)<\infty$. Another two out of three argument using the facts $P_A(M_i,A),P_A(M_i,M)<\infty$ shows that $P_A(M_i,M_i)<\infty$ for all $i\geq 0$. Now choose a sppj resolution $P$ of $M$ such that $\sup(M_i)=\sup(M)=0$ for all $i\geq 0$. This is always possible since we can choose $\H^{\sup}(f_i)$ to have a non-trivial kernel and the exact sequence $\H^{\sup(M_i)}(M_{i+1})\to \H^{\sup(M_i)}(P_i)\to \H^{\sup(M_i)}(M_i)\to 0$ gives the desired property. We now claim that there exists $n\geq 0$ such that $\sup\RHom_A(M_n,A)=0$. If $\sup \RHom_A(M,A)=0$, we are done. If $\sup \RHom_A(M,A)<0$, then applying $\RHom_A(-,A)$ to the triangle 

    $$
M_{1} \xra{g_{1}} P_0 \xra{f_0} M_{0} \to M_1[1]$$

shows that $\sup \RHom_A(M_1,A)=0$. If $\sup \RHom_A(M,A)>0$, the same process shows that $\sup \RHom_A(M,A)>\sup \RHom_A(M_1,A)$. Repeating this process for $i\geq 1$ sufficiently many times yields an integer $n$ such that $\sup \RHom_A(M_n,A)=0$. Fix such an integer $n$. Now for each integer $i\geq 1$, applying $\RHom_A(M_n,-)$ to the triangle 

 $$
M_{n+i+1} \xra{g_{n+i+1}} P_{n+i} \xra{f_{n+i}} M_{n+i} \to M_{n+i+1}[1]$$

shows that $\Ext^1_A(M_n,M_{n+1})\simeq \Ext^i(M_n,M_{n+i})$. If $\Ext^1_A(M_n,M_{n+1})\neq 0$, then $B(M_n)=\infty$. Thus, it must be the case that $\Ext^1_A(M_n,M_{n+1})=0$. Therefore the triangle

 $$
M_{n+1} \xra{g_{n+1}} P_{n} \xra{f_{n}} M_{n} \to M_{n+1}[1]$$

splits in $\D(A)$ and $M_n\in \mathcal{P}$. By \cite[Theorem 1.1]{minamoto}, $\projdim_A(M)<\infty$.

\end{proof}

\begin{prop}\label{prop:Gor_auslander_bounds}
    Let $A$ be as in \Cref{setup}. Then the following are equivalent:
    
    \begin{enumerate}
        \item $A$ is Gorenstein.

        \item There exists $0\not\simeq M\in \Db(A)$ such that $\injdim_A(M)<\infty$, $B(M)<\infty$ and $P_A(M,A)<\infty$.
    \end{enumerate}
    
\end{prop}

\begin{proof}
    For (1) $\implies$ (2), take $M=A$. For (2) $\implies$ (1), \Cref{lem:fpd} implies $\projdim_A(M)<\infty$ and hence $M\in \langle A\rangle$ (see \cite[Proposition 2.26]{minamoto}). This implies $M$ is reflexive and \cite[Proposition 4.1]{hu2023gdimensionsdgmodulescommutativedgrings} implies $\mathrm{Gdim}_A(M)<\infty$. \Cref{prop:gdim} then implies $A$ is Gorenstein, finishing the proof.
\end{proof}

\textbf{Proof of (1) $\iff$ (5).}

\begin{prop}[cf. \cite{DGI}, Proposition 4.5]
    Let $A$ be as in \Cref{setup}. Let $k$ denote the residue field of $\H^0(A)$. Then the following are equivalent:

    \begin{enumerate}
        \item $A$ is Gorenstein.

        \item For every $M\in \Db(A)$ such that $\injdim_A(M)<\infty$, there exists $0\not \simeq N_M\in \langle M\rangle \cap \langle k\rangle $ such that $B(N_M)<\infty$ and $P_A(N_M,A)<\infty$.
    \end{enumerate}
\end{prop}

\begin{proof}
    For (1) implies (2), let $\underline{x}$ denote a minimal generating set for the maximal ideal of $\H^0(A)$. For every $M\in \Db(A)$ such that $\injdim_A(M)<\infty$, we will show that we can take $N_M=A//\underline{x}$. Since $A$ is Gorenstein and $\injdim_A(M)<\infty$, we have $\projdim_A(M)<\infty$ by \cite[Theorem 4.5]{local_CM_modules}. Since $M$ and $A//\underline{x}$ are compact objects of $\D(A)$ and $\mathrm{Supp}(A//\underline{x})\subseteq \mathrm{Supp}(M)$, $A//\underline{x}\in \langle M\rangle$ by \cite[Theorem 4.16]{lifting_stratifications_ttcategories}. Since $\H^n(A//\underline{x})\in \langle k\rangle $ for all $n$, $A//\underline{x}\in \langle k\rangle$ by \cite[3.10]{DGI}. Finally, $P_A(A//\underline{x},A)<\infty$ since $\injdim_A(A)<\infty$ and $B(A//\underline{x})=\projdim_A(A//\underline{x})<\infty$.

    \par For (2) implies (1), as in the proof of \Cref{prop:G_0}, there exists $B\in \Db(A)$ such that $\injdim_A(B)<\infty$. Since $\injdim_A(N_B)<\infty$, applying \Cref{prop:Gor_auslander_bounds} with $M=N_B$ finishes the proof.
\end{proof}

\textbf{Proof of (1) $\iff$ (6).}

\begin{lem}\label{lem:MCM_sup}
    Let $A$ be a dg-algebra as in \Cref{setup}. If $M\in \mathrm{MCM}(A)$ with $\sup(M)=0$ and $\injdim_A(M)<\infty$, then $\sup \RHom_A(M,M)=0$.
\end{lem}

\begin{proof}

By \cite[Lemma 4.6]{local_CM_modules}, we have $\sup\RHom_A(M,M)=\injdim_A(M)-\depth_A(M)$. By the Bass formula, \cite[Theorem 2.33]{Minamoto2}, we have $\injdim_A(M)=\depth_A(M)$, so that \\
$\sup\RHom_A(M,M)=0$.

\end{proof}

\begin{prop}
    Let $A$ be a dg-algebra as in \Cref{setup}. Then the following are equivalent:

    \begin{enumerate}
        \item $A$ is Gorenstein.

        \item There exists $0\not\simeq M\in \mathrm{MCM}(A)$ such that $\injdim_A(M)<\infty$, $\sup(M)=\sup \RHom_A(M,A)=0$ and $B(M)<\infty$.
    \end{enumerate}
\end{prop}

\begin{proof}
    For the forward implication, $M=A$ works by \cite[Theorem 2 and 3]{Shaul_CM}. For the backward implication, choose an sppj resolution $P$ of $M$ as in \Cref{def:sppj}. By \Cref{lem:MCM_sup}, $\sup \RHom_A(M,M)=0$. Combining this with $\sup\RHom_A(M,A)=0$, we see that $\sup \RHom_A(M,M_1)\leq 1$. For each $i\geq 0$, we have

    \[\sup\RHom_A(M,A[-\sup(M_i)])=\sup\RHom_A(M,A)+\sup(M_i)\leq \sup\RHom_A(M,A)=0. \]

    Thus for each $i\geq 0$, we have $\sup\RHom_A(M,P_i)\leq \sup\RHom_A(M,A)=0$. This implies that for each $i\geq 1$, we have $\Ext_A^1(M,M_1)\simeq \Ext_A^i(M,M_i)$. Suppose $\Ext_A^1(M,M_1)\neq 0$. We then have for $n>i$, $\sup\RHom_A(M,M_i)<\sup\RHom_A(M,M_n)$. Therefore 

    \[P_A(M,M_i)=\sup \RHom_A(M,M_i)-\sup(M_i)<\sup \RHom_A(M,M_{i+1})-\sup(M_{i+1})=P_A(M,M_{i+1})\]

which implies $\lim P_A(M,M_i)=\infty$. This contradicts the assumption that $B(M)<\infty$. Thus $\Ext_A^1(M,M_1)= 0$ and the triangle

   $$
M_1 \xra{} P_0 \xra{} M \to M[1]$$

splits, so that $\projdim_A(M)<\infty$. \cite[Proposition 4.1]{hu2023gdimensionsdgmodulescommutativedgrings} implies that $\mathrm{Gdim}_A(M)<\infty$ and \Cref{prop:gdim} then implies that $A$ is Gorenstein, finishing the proof.
    
\end{proof}

\textbf{Proof of (1) $\iff$ (7).}

\begin{lem}\label{lem:koszulcomplex}
    Let $A$ be as in \Cref{setup}. Let $M\in \Db(A)$ and $G \xra{\simeq} M$ a semifree resolution. Assume there exists an integer $n$ such that $G/F^nG$ has cohomology concentrated in a single degree. If $N\in \Db(A)$ is such that $\RHom_A(M,N)\in \Db(A)$, then $\RHom_A(\H^0(A),N)\in \Db(A)$. If in addition $\H^0(A)$ is a field, then $\injdim_A(N)<\infty$.

\end{lem}

\begin{proof}
Consider the exact sequence

   $$
0 \xra{} F^nG \xra{} G \to G/F^nG \xra{} 0 .$$

Since $G/F^nG$ has cohomology concentrated in a single degree, using smart truncations as in \Cref{remark:smart}, we see that $G/F^nG$ is isomorphic in $\D(A)$ to an object of the form $L[-a]$ where $L$ is an $\H^0(A)$-module. Moreover, since $\RHom_A(F^nG, N)$ and $\RHom_A(G,N)$ are cohomologically bounded, so is $\RHom_A(L,N)$. By adjunction, we see that $\RHom_A(L,N)\simeq \RHom_{\H^0(A)}(L,\RHom_A(\H^0(A),N))$. Since $L$ is free over $\H^0(A)$, $\RHom_{\H^0(A)}(L,\RHom_A(\H^0(A),N))$ is isomorphic in $\D(A)$ to a product of copies of $\RHom_A(\H^0(A),N)$. Thus $\RHom_A(\H^0(A),N)\in \Db(A)$ as well. Finally, if $\H^0(A)$ is a field, \cite[Corollary 2.31]{Minamoto2} implis that $\injdim_A(N)<\infty$.

\end{proof}

\begin{prop}
      Let $A$ be as in \Cref{setup}. Let $\underline{x}=x_1,\dots,x_n$ be a generating set for the maximal ideal of $\H^0(A)$, and set $K:=A//\underline{x}$. Then the following are equivalent.

      \begin{enumerate}
          \item $A$ is Gorenstein.

          \item There exists $0\not\simeq M\in \Db(K)$ with semifree resolution $G\xra{\simeq} M$ such that 

          \begin{enumerate}
              \item $G/F^nG$ has cohomology concentrated in a single degree for some integer $n$.

              \item $\RHom_K(M,K)\in \Db(K)$.
          \end{enumerate}
      \end{enumerate}
\end{prop}

\begin{proof}
    Assume $A$ is Gorenstein. By \cite[Theorem 4.11]{Shaul_Koszul} (see also \cite[Theorem 4.9]{Frankild2003}), $K$ is Gorenstein. By \Cref{def:Gorenstein}, $\RHom_K(k,K)\in \Db(K)$. Clearly $G/F^{-1}G$ has cohomology concentrated in a single degree and thus the proof of the forward implication is complete. Now assume the conditions of (2). By \Cref{lem:koszulcomplex}, $\injdim_K(K)<\infty$ i.e. $K$ is Gorenstein. By \cite[Theorem 4.11]{Shaul_Koszul}, $A$ is Gorenstein and the proof is complete.
\end{proof}

\textbf{Proof of (1) $\iff$ (3).}  For (1) $\implies$ (3), take $M=A$. It is clear that (3) $\implies$ (4) and we have shown that (4) implies (1).

\bibliographystyle{amsalpha}
\bibliography{references}
\Addresses
\end{document}